\documentclass[11pt,letterpaper]{amsart}
\usepackage{amsmath,amscd}
\usepackage{amsfonts,latexsym,rawfonts,amsmath,amssymb,amsthm}
\usepackage[usenames,dvipsnames]{xcolor}

\numberwithin{equation}{section}
\newtheorem{prop}{Proposition}[section]
\newtheorem{theorem}[prop]{Theorem}
\newtheorem{lemma}[prop]{Lemma}
\newtheorem{corollary}[prop]{Corollary}
\newtheorem{remark}[prop]{Remark}
\newtheorem{example}[prop]{Example}
\newtheorem{definition}[prop]{Definition}

\newtheorem{problem}[prop]{Problem}
\newcommand{\R}{\mathbb{R}}
\newcommand{\Lip}{\operatorname{Lip}}

\begin{document}

\title[New Asymptotic Geometric Quantities]{New Asymptotic Geometric Quantities in Riemannian Geometry and Their Geometric and Dynamical Applications}
\author{Xiaoshang Jin and Jiabin Yin}
\thanks{X. J. is supported by ``the Fundamental Research Funds for the Central Universities'', HUST: \# 2025BRSXB002, NSFC
(Grant No. 12471054) and J. Y. is supported by the NSFC (Grant No. 12201138), Mathematics Tianyuan fund project (Grant No. 12226350) and NSF of Henan Province (Grant No. 262300421869)}
\date{}

\maketitle

\begin{abstract}
We introduce large $p$ asymptotic geometric quantities associated with
$p$-capacity, the first $p$-eigenvalue, and the Maz'ya constant on complete
noncompact Riemannian manifolds.  We prove the hierarchy
$$
        \mathcal{V}(M)\geq \mathcal C(\Omega)\geq \Lambda(M)=\mathcal M(M)\geq0,
$$
and show that, under a centered-ball isoperimetric condition or a
rotational symmetry condition, these quantities coincide with the volume
entropy or the dimension.

 In the Hadamard nonpositively curved
case it also agrees with the topological entropy of the geodesic flow. As an application, combining with the entropy rigidity theorem, we obtain a characterization of hyperbolic manifolds.

We also prove a second-order refinement.  For a Hadamard manifold with compact quotient, under certain condition,
the first-order large $p$ capacitary limit detects volume entropy,
whereas the logarithmic second-order correction detects the rank.



\end{abstract}

\section{Introduction}

\begin{definition}
  Assume that $(M,g)$ is a complete Riemannian manifold and $\Omega\subseteq M$ is a compact subset. For any $p>1,$
the $p-$capacity of $\Omega$ in $M$ is defined as the infimum of the $p-$Dirichlet energy:
\begin{equation*}
\begin{aligned}
\mathrm{Cap}_p(\Omega)=&\mathrm{Cap}_p(\Omega, M)=\inf\left\{\int_{M}|\nabla u|^p \,dv_g:\ u|_{\Omega}=1, \ u\in C^\infty_0(M)\right\}.
\end{aligned}
\end{equation*}
\end{definition}
In general, we require that $M$ is non-compact and $\Omega$ is a closed domain with compact smooth boundary $\partial\Omega.$ Then
\begin{equation*}
\begin{aligned}
\mathrm{Cap}_p(\Omega)=&\inf\Big\{\int_{M\setminus\Omega}|\nabla u|^p\, dv_g:\ u \ \text{is locally Lipschitz in}\\
&  M\setminus\Omega,\ u|_{\partial\Omega}=1, \ \lim\limits_{x\rightarrow\infty}u(x)=0\Big\}.
\end{aligned}
\end{equation*}
The infimum in the definition of $p-$capacity is attained by the unique function
$u$ that satisfies the following Dirichlet problem:
The infimum energy can be achieved by the unique solution $u_p$ to the following PDE:
$$\begin{cases}
  \Delta_p u={\rm div}(|\nabla u|^{p-2}\nabla u)=0\ & {\rm in}\ M\setminus\Omega, \\
  u=1\ &{\rm on}\ \ \partial\Omega,\\
  u(x)\rightarrow0  &{\rm as} \ \ x\rightarrow\infty.
\end{cases}$$
Such a function $u_p$ is called the $p-$capacitary potential of $\Omega$ in $M.$ However, the solution may not exist unless manifold admits a $p-$Green function that vanishes at infinity, see \cite[Theorem 4.1]{fogagnolo2022minimising}. When it exists, the $p-$capacity is given by
$$
\mathrm{Cap}_p(\Omega)=\int_{M\setminus\Omega}|\nabla u_p|^p\, dv_g=\int_{\{u=t\}}|\nabla u_p|^{p-1}\, d\sigma_g
$$
for any regular value $t\in(0,1].$ Here $d\sigma_g$  denotes the ($n-$dimensional) Riemannian surface element induced on the level set.
\par The concept of capacity originates from electrostatics and $\mathrm{Cap}_2(\Omega,\mathbb{R}^3)$ represents the total electrical charge the conductor $\Omega$ can hold. For a general $p>1$, the $p$-capacity is a fundamental notion in nonlinear potential theory and geometric analysis.
In Euclidean space $\mathbb{R}^n$, sharp relationships between capacity, volume, surface area, and mean curvature have been extensively studied. For instance, \cite{adams2025full,kruglikov1987capacity,maz2013sobolev,polya1951isoperimetric,xiao2017p} etc.  More recently, there are some sharp upper and lower bounds for the $p$-capacity in hyperbolic spaces using inverse mean curvature flow, conformal methods, one can see \cite{jin2026sharp,jin2025sharply,li2025sharp}for more details.
\par In general Riemannian manifolds, \cite{grigor1999isoperimetric} provided some basic conceptions and tools of capacity. In 2008, Bray and Miao \cite{bray2008capacity}  established a sharp upper bound for $\mathrm{Cap}_2(\Omega)$
 in an asymptotically flat 3-manifold with nonnegative scalar curvature (other related mass-capacity inequality see \cite{miao2024implications,miao2025mass,oronzio2025area,oronzio2025adm}).  The result links the capacity to the Hawking mass and the total ADM mass and was later extend to general $p\in(1,3)$ case by Xiao in \cite{xiao2016p}. The theory was further developed to mass-$p$-capacity in \cite{mazurowski2023monotone,xia2024new} (other relevant literature
see \cite{yin2026sharp}). For asymptotically hyperbolic Einstein manifolds, the first author studied the capacity of balls in \cite{jin2025relative}.
\par In this paper, we investigate relations between volume entropy of $M$ and the asymptotic behavior of three families of geometric quantities on a complete noncompact Riemannian manifold: the $p$-capacity of a compact set, the first Dirichlet $p$-eigenvalue, and the Maz'ya constant, as $p\to\infty$.

To prepare for the statement of our main results, we begin by introducing several necessary definitions and outlining the motivation for studying these asymptotic geometric quantities.
\\~
\par  In the following, we use $|\cdot|$ to denote the volume of a set with respect to the Hausdorff measure of the corresponding dimension.
\par Let $(M,g)$ be a complete non-compact Riemannian manifold and $\Omega$ be a compact subset of $O,$ where $O$ is a bounded domain in $M.$ Then the infinity capacity is well defined and satisfies the following property (see Section \ref{Appendixsection}):
$$
{\rm Cap}_{\infty}(\Omega,O):=\lim\limits_{p\rightarrow\infty} \Big({\rm Cap}_{p}(\Omega,O)\Big)^{\frac 1p}=\frac 1{d_g(\partial \Omega,\partial O)},
$$
where ${\rm Cap}_{p}(\Omega,O)$ is the relative $p$-capacity of $(\Omega,O)$ defined as
$$
{\rm Cap}_p(\Omega,O)=\inf\left\{\int_{O\setminus\Omega}|\nabla u|^p\,d\mu : u\in \mathrm{Lip}_c(O),\; u\ge 1\text{ on }\Omega\right\}.
$$
The infinity capacity admits a geometric interpretation as the reciprocal of the distance. However, when considering the capacity of $\Omega$ in the case $O=M,$ the above definition becomes meaningless as it yields zero. Therefore, we introduce a new definition of the infinity capacity of $\Omega$ when
$O=M.$
   \begin{definition}
Let $\Omega\subseteq M$ be a bounded domain, then we call
$$
\mathcal{C}(\Omega)=\limsup\limits_{p\rightarrow\infty}p\cdot\Big({\mathrm{Cap}_p(\Omega)}\Big)^{\frac 1p}
$$
the infinity capacity of $\Omega.$
\end{definition}
Here we define the infinity capacity using the limit superior because the limit may not exist. In fact, we can construct an example to show that
$$\limsup\limits_{p\rightarrow\infty}p\cdot\Big({\mathrm{Cap}_p(\Omega)}\Big)^{\frac 1p}>\liminf\limits_{p\rightarrow\infty}p\cdot\Big({\mathrm{Cap}_p(\Omega)}\Big)^{\frac 1p}$$
for $\Omega$ in some manifold $(M,g).$ See Example \ref{ex3.2}.
\par
We also claim that $\mathcal{C}(\Omega)=0$ for all domains in Euclidean space $\mathbb{R}^{n+1}$ as ${\mathrm{Cap}_p(\Omega)}=0$ for all $p\geq n+1$
and that $\mathcal{C}(\Omega)=n$ for all domains in hyperbolic space $\mathbb{H}^{n+1}.$
\\~\\
\par Assume that $\lambda_{1,p}(M)$ is the $p-$eigenvalue of $M$ and defined by
$$
\lambda_{1,p}(M)=\lim\limits_{k\rightarrow\infty}\lambda_{1,p}(\Omega_k)=\inf\left\{\frac{\int_M |\nabla u|^p\, dv_g}{\int_M |u|^p\,dv_g}:\ u\neq 0,\ u\in C^\infty_0(M)\right\}
$$
for any smoothly compact exhaustion $\{\Omega_k\}_{k=1}^\infty$ of $M.$
The first Dirichlet eigenvalue $\lambda_{1,p}(\Omega)$ is defined as
the least positive number $\lambda$ such that the following equation has a non-zero weak solution.
$$
\begin{cases}
  \Delta_p u=-\lambda |u|^{p-2} u\ &\ {\rm in}\ \Omega,
 \\ u=0 & {\rm on}\ \partial\Omega
\end{cases}$$
where  $\Delta_p u={\rm div}(\|\nabla u\|^{p-2}\nabla u)$ is called the $p-$Laplacian.

If $\Omega\subset M=\mathbb R^n$, Juutinen-Lindqvist-Manfredi \cite{juutinen1999eigenvalue} proved that
$$
\lambda_{1,\infty}(\Omega):=\lim_{p\to\infty}\Big(\lambda_{1,p}(\Omega)\Big)^{1/p}=\frac{1}{\max\{ {\rm dist}(x,\partial\Omega):\ x\in\partial\Omega\}}.
$$
If $\Omega=M$ is a complete non-compact Riemannian manifold, we can see that $\lambda_{1,\infty}(M)=0$.  Therefore, we introduce a new definition of the infinity for $\lambda_{1,p}(M).$

  \begin{definition}
Let $(M,g)$ be a complete non-compact Riemannian manifold, then we call
$$
\Lambda(M)=\lim\limits_{p\rightarrow\infty}p\cdot\Big(\lambda_{1,p}(M)\Big)^{\frac 1p}
$$
the infinity eigenvalue of $M.$
\end{definition}
Recall an interesting result from in \cite[Lemma 2.4]{jin2025lower} that
$p\cdot\Big(\lambda_{1,p}(\Omega)\Big)^{\frac 1p}$ and $p\cdot\Big(\lambda_{1,p}(M)\Big)^{\frac 1p}$ are always increasing in $p.$ Hence $\Lambda(M)$ is well defined.
\\~\\
\par Recall the Maz'ya constant $m_p(M)$ is defined as
$$
m_p(M) = \inf_{F\subset\subset M} \frac{{\rm Cap}_p(F,M)}{|F|}
$$
\begin{definition}
The Maz'ya limit of a complete noncompact Riemannian manifold $M$ is defined as
$$
\mathcal{M}(M) := \lim_{p\to\infty} p\, \Big(m_p(M)\Big)^{1/p}.
$$
\end{definition}
Since $\Lambda(M)$ is well defined, it follows from Theorem \ref{thm1.6} (1) that $\mathcal M(M)$ is also well defined.
\\~\\
\par In the end, we recall the definition of the volume entropy for a non-compact manifold:
 \begin{definition}
Let $o\in M$, then we call
$$
\mathcal{V}(M)=\limsup_{R\rightarrow\infty}\frac{\ln |B(o,R)|}{R}
$$
the volume entropy of $M.$
\end{definition}
The definition is independent of the choice of the base point $o.$ In fact, for any different points $o$ and $o'$ in $M$ with distance $d,$
we have $B(o,R)\subseteq B(o',R+d)$ and hence
$$
\mathcal{V}_o(M)=\limsup_{R\rightarrow\infty}\frac{\ln |B(o,R)|}{R}\leq \limsup_{R\rightarrow\infty}\frac{\ln |B(o',R+d)|}{R}=\mathcal{V}_{o'}(M).
$$ The reverse inequality also holds by symmetry.

\par In this paper, we study the property of $\mathcal{C}(\Omega),$ especially its relationship with $\Lambda(M)$, volume entropy $\mathcal V(M)$ and Maz'ya limit $\mathcal M(M)$:

 \begin{theorem}\label{thm1.6}
  Let $(M,g)$ be a complete non-compact Riemannian manifold.
  \begin{itemize}
  \item [(1)] For all compact smooth domains $\Omega$ in $M,$
    \begin{equation}\label{eq1.1}
     \mathcal{V}(M)\geq   \mathcal{C}(\Omega)\geq \Lambda(M)=\mathcal M(M)\geq 0.
    \end{equation}
    \item [(2)] If $(M,g,\Omega,o)$ satisfies one of the following conditions:
    \begin{itemize}
      \item [(A)] $(M,g)$ obeys the isoperimetry of $o$-centered balls for $o\in M,$ i.e.
     $$
     \forall\ compact\ smooth \ K \  with \ |K|=|B(o,r)|\Rightarrow|\partial K|\geq |\partial B(o,r)|,
     $$ and $\Omega$ is a compact smooth domain in $M;$
      \item [(B)] $(M,g)=([0,\infty),\ dt^2+\varphi(t)^2g_{\mathbb{S}})$ is a rotationally symmetric manifold with center $o=\{t=0\}$ and $\Omega$ is a compact smooth domain in $M$ containing $o,$
    \end{itemize}then $\mathcal{C}(\Omega)$ is a constant independent of $\Omega.$
     Moreover, if in addition,
     \begin{equation}\label{eq1.2}
      \frac{|\partial B(o,t)|}{|B(o,t)|}\ \ \text{is decreasing for sufficiently big}\ t,
     \end{equation}
     then \begin{equation}\label{eq1.3}
     \mathcal{V}(M)=   \mathcal{C}(\Omega)=\Lambda(M)=\mathcal M(M).
     \end{equation}
   \end{itemize}
  \end{theorem}
  \begin{remark}
  \begin{itemize}
    \item The inequality in case \eqref{eq1.1} does not necessarily attain equality; in fact, we can provide some examples where the strict inequalities hold, as shown in Example \ref{ex3.1} and Example \ref{ex3.2}.
    \item From Theorem \ref{thm1.6}, it is natural to pose the following question:
    \begin{problem}
      For any noncompact manifold, does $\mathcal{C}(\Omega)$ depend on $\Omega?$
    \end{problem}
  \end{itemize}

  \end{remark}
    Before we introduce the following result, we observe that $\mathcal{C}(\Omega)$  scales as follows under a change of metric:
  $$
  \mathcal{C}(\Omega)_{\lambda^2 g}=\limsup\limits_{p\rightarrow\infty}p\cdot\big[{\mathrm{Cap}_p(\Omega)_{\lambda^2 g}}\big]^{\frac 1p}
  =\limsup\limits_{p\rightarrow\infty}p\cdot\big[\lambda^{n-p}{\mathrm{Cap}_p(\Omega)_{g}}\big]^{\frac 1p}=\lambda^{-1}\mathcal{C}(\Omega)_{g}
  $$ for any $\lambda>0.$
  Similarly, one can check that $\Lambda(M)$ and $\mathcal{V}(M)$ are also scale covariant with exponent $-1.$
  \par Here are some estimates of $\mathcal{C}(\Omega)$ under certain curvature conditions.
    \begin{theorem}\label{thm1.9}
    Let $(M,g)$ be a complete non-compact Riemannian manifold of dimension $n+1$ and $\Omega$ be a bounded compact domain in $M.$
      \begin{itemize}
      \item [(1)]   If ${\rm Ric}[g] \geq 0,$ then $\mathcal{C}(\Omega)=0$ for all $\Omega.$
    \item [(2)]  If ${\rm Ric}[g]\geq-ng,$ then $\mathcal{C}(\Omega)\leq \mathcal V(M)\leq n$  for all $\Omega.$
    \item [(3)] If ${\rm Ric}[g]\geq-ng$ and there exists an exhaustion $\{\Omega_k\}_{k=1}^\infty$  of $M$ with mean curvature $H_k$ of $\partial\Omega_k$ such that $H_k\geq n,$
        then $\mathcal V(M)=\mathcal{C}(\Omega)=\Lambda(M)=n$  for all $\Omega.$
  \item [(4)] If $(M,g)$ is a conformally compact Einstein manifold with conformal infinity of non-negative Yamabe constant, then $\mathcal V(M)=\mathcal{C}(\Omega)=\Lambda(M)=n$  for all $\Omega.$
  \end{itemize}
\end{theorem}

\par In the following, we study the application of Theorem \ref{thm1.6} on the topological entropy of close manifolds, which is a core concept in topological dynamical systems \cite{Knieper1997}\cite{Ma1979}. In particular, we provide several geometric and dynamical applications of the newly introduced asymptotic geometric quantities.
\begin{theorem}\label{thm:1.10}
Let $N$ be a closed manifold of dimension $n+1$ and $(M,g)$ be its  universal cover. Assume that $(M,g)$ is a non-flat Hadamard manifold and $(M,g,\Omega,o)$ satisfies satisfies one of the conditions $(A)$ and $(B)$ in Theorem \ref{thm1.6} (2),
Then
\begin{equation}\label{eq1.4}
  \mathcal V(M)=\mathcal{C}(\Omega)=\Lambda(M)=\mathcal M(M)=h_{top}(\phi),
\end{equation}
\end{theorem}

As an application of Theorem \ref{thm:1.10}, we obtain the following rigidity result.

\begin{corollary}[Hyperbolic Rigidity]\label{thm:1.11}
Let $(N,g)$ be a closed  $n+1$-dimensional Riemannian manifold with ${\rm Ric}[g]\geq-ng$. If there exists $p>1$ such that the first eigenvalue of  the universal cover $M$ satisfies
$$\lambda_{1,p}(M)\geq\left(\frac{n}p\right)^p.$$
Then $N$ is isometric to a compact hyperbolic manifold, and consequently $M$ is isometric to the hyperbolic space $\mathbb{H}^{n+1}.$
\end{corollary}

\begin{remark}
\begin{itemize}
\item  This corollary generalizes Ledrappier–Wang's result \cite{LW2010} for $p=2$ to  $p>1$ and can be seen as the counterpart of  Peigné-Sambusetti's result \cite{peigne2019entropy}.
\item  Hyperbolic rigidity still holds if the eigenvalue condition is replaced by condition (3) or (4) of Theorem \ref{thm:1.10}.
\end{itemize}
\end{remark}

\medskip

Next, we now point out that the large $p$ capacitary quantities introduced above
also contain a second-order asymptotic information.  The results proved so
far are first-order entropy results.  More precisely, under the geometric
alternatives in Theorem~1.6, Theorem~1.10 gives
\[
        \mathcal V(M)=\mathcal{C}(\Omega)=\Lambda(M)=M(M)=h_{\mathrm{top}}(\phi).
\]
Equivalently, at the first-order level, the two normalizations
\[
       \lim_{p\rightarrow\infty} p\,{\rm Cap}_p(\Omega)^{\frac1p}
        =\lim_{p\rightarrow\infty}
        (p-1){\rm Cap}_p(\Omega)^{\frac{1}{p-1}}=\mathcal V(M).
\]

The next question is whether this capacitary family sees geometric
information beyond the entropy.  This question is natural in nonpositive
curvature.  Indeed, by Knieper's asymptotic sphere-growth theorem, if $M$
is a nonflat Hadamard manifold with nonpositive sectional curvature and
compact quotient, then
\[
        |\partial B(o,r)|
        \sim
        r^{\frac{(\operatorname{rank}M-1)}{2}}e^{\mathcal V(M) r},
        \qquad r\to\infty
\]
where “$\sim$” denotes infinity of the same order.
 Thus the first-order entropy detects the exponential
factor $e^{\mathcal V(M) r}$, whereas the rank is encoded in the polynomial correction.
Here the rank of a Hadamard manifold $M$ is defined (see \cite{ballmann1985structure}) by
$$
\operatorname{rank}M=\min\big\{\operatorname{rank}(v):v\in SM\big\}
$$
and $\operatorname{rank}(v)$ is the dimension of the space of all
parallel Jacobi fields along the geodesic $\gamma_v(t)$ witch has initial velocity $v.$

This motivates the following second-order quantity.  Assume $\mathcal V(M)>0$ and define
\begin{equation}\label{eq1.5}
        \alpha_{\operatorname{cap}}(\Omega)
        :=
        \lim_{p\to\infty}
        \left(\frac{
        (p-1){\rm Cap}_p(\Omega)^{\frac{1}{p-1}}
        }{
       \mathcal V(M)
        }-1\right)\cdot\frac{p}{\ln p},
\end{equation}
whenever the limit exists.  The normalization by $p-1$ is not meant to
change the first-order entropy limit.  Rather, it removes a universal
second-order correction and isolates the geometric polynomial exponent.

The following theorem is a second-order refinement of Theorem~1.10.

\begin{theorem}
\label{thm:second-order-refinement-110}
Under the assumptions of Theorem~\ref{thm:1.10}, as $p\to\infty$,
\begin{equation}\label{eq1.6}
       \frac{(p-1){\rm Cap}_p(\Omega)^{\frac{1}{p-1}}}{\mathcal V(M)}
        =1 + \left(\frac{\operatorname{rank}M-1}{2}\right)\frac{\ln p}{p}+O\!\left(\frac1p\right).
\end{equation}
Consequently,
\begin{equation}\label{eq1.7}
        \operatorname{rank}M
        =
        1+2\alpha_{\operatorname{cap}}(\Omega).
\end{equation}
\end{theorem}

\begin{corollary}
\label{cor:second-order-consequences}
Under the assumptions of Theorem~\ref{thm:second-order-refinement-110},
the normalization used in the definition of $\mathcal C(\Omega)$ satisfies
\begin{equation}\label{eq1.8}
 \operatorname{rank}M
        =
        -1+
        2
        \lim_{p\to\infty}
        \left(\frac{
        p{\rm Cap}_p(\Omega)^{\frac1p}
        }{
        \mathcal V(M)
        }-1\right)\cdot\frac{p}{\ln p}
.
\end{equation}
In particular, $M$ has rank one if and only if
$$
        (p-1){\rm Cap}_p(\Omega)^{\frac{1}{p-1}}
        =
        \mathcal V(M)+O\!\left(\frac1p\right).
$$

Moreover, the second-order invariant separates spaces with the same
entropy but different ranks.  Indeed, if $M_1$ and $M_2$ both satisfy the
assumptions of Theorem~\ref{thm:second-order-refinement-110} and
$$
        \mathcal V(M_1)=\mathcal V(M_2),
        \qquad
        \operatorname{rank}M_1\neq \operatorname{rank}M_2,
$$
then, for any admissible compact smooth domains $\Omega_i\subset M_i$,
$$
        \alpha_{\operatorname{cap}}(\Omega_1)
        \neq
        \alpha_{\operatorname{cap}}(\Omega_2).
$$
\end{corollary}

In summary, the large \(p\) quantities introduced in this paper,
namely the infinity capacity \(\mathcal C(\Omega)\), the infinity
eigenvalue \(\Lambda(M)\), and the Maz'ya limit \(\mathcal M(M)\), provide
potential-theoretic invariants associated with the Riemannian metric and
the \(p\)-Laplacian.  Under the geometric hypotheses considered above,
these invariants recover the volume entropy; on universal covers of closed
nonpositively curved manifolds, this common value agrees with the
topological entropy of the geodesic flow.

The second-order refinement shows that the capacitary family contains
information beyond the first-order entropy.  In the compact-quotient
nonpositively curved setting, the logarithmic correction to the large \(p\)
capacity detects the polynomial correction in the growth of geodesic
spheres, and hence the rank.  Thus large \(p\) capacitary invariants provide
analytic probes of geometric and dynamical features which are not visible
from entropy alone.

Finally, the examples below show that the equalities
\[
        \mathcal C(\Omega)=\Lambda(M)=\mathcal M(M)=\mathcal V(M)
\]
are not automatic on general noncompact manifolds.  Irregular radial volume
growth may lead to strict inequalities, and may also force one to use
\(\limsup\) in the definition of \(\mathcal C(\Omega)\).  This suggests
that the large \(p\) invariants retain finer asymptotic information than
the volume entropy alone.

\par The paper is organized as follows. Section~2 studies the infinity capacity $\mathcal C(\Omega)$ and proves Theorems~\ref{thm1.6} and \ref{thm1.9}. Section~3 applies these results to universal covers of closed manifolds, identifies the large $p$ quantities with entropy under the hypotheses of Theorem~\ref{thm:1.10}, and derives the corresponding rigidity result. Section~4 proves the second-order capacitary refinement and the rank-detection formula. Section~5 presents examples showing that the equalities in \eqref{eq1.1} may fail in general. Finally, the appendix proves that the infinity capacity of a condenser equals the reciprocal of the distance between the boundary components, using the equivalence between infinity harmonic functions and absolutely gradient minimizing extensions on Riemannian manifolds.

\section{Proof of Theorem 1.6 and 1.9}
Before we prove Theorem 1.6, we will first provide a well-known result concerning the estimate of capacity: upper bound via flux and lower bound via the isoperimetric function .
 \begin{lemma}[\cite{polya1951isoperimetric,grigor1999isoperimetric}]
Let $(M,g)$ be a complete and non-compact Riemannian manifold and $o\in M.$  Then for any $p>1$ and $r>0,$
 \begin{equation}\label{eq2.1}
   {\rm Cap}_p(B(o,r))\leq \left(\int_r^\infty |\partial B(o,t)|^{\frac{1}{1-p}}\,dt\right)^{1-p}.
 \end{equation}
 If $M$ admits an isoperimetric function $I(\cdot),$
 then for any bounded domain $\Omega,$
 \begin{equation}\label{eq2.2}
 {\rm Cap}_p(\Omega)\geq\left(\int_{|\Omega|}^{|M|} \big[I(\tau)\big]^{\frac{p}{1-p}}\,d\tau\right)^{1-p}.
\end{equation}
 Here $I(\tau)$ is defined by
 $$
 I(\tau)=\inf\Big\{|\partial\Omega|:\ \Omega\subseteq M\ \text{is precompact with smooth boundary and}\ |\Omega|=\tau\Big\}
 $$
 \end{lemma}

 \begin{proof}[Proof of Theorem 1.6]
 If $|M|<\infty,$ then ${\rm Cap}_p(\Omega)=0$ for all compact domain $\Omega$ and $p>1$ ($M$ is called $p-$parabolic). This is obvious due to \eqref{eq2.1} and hence we will always assume that the volume of $M$ is infinity.
   \begin{itemize}
   \item [(1)] {\bf Step 1: $\mathcal{C}(\Omega)\geq\Lambda(M)=\mathcal M$.}
   \par Recall the definition of the Maz'ya constant in \cite{maz2013sobolev} or (19) and (20) in \cite{grigor1999isoperimetric}, for any bounded domain $O$ and $p>1,$
   $$
   \lambda_{1,p}(O)\leq m_p(O):=\inf\limits_{F\subset\subset O}\frac{{\rm Cap}_p(F,O)}{|F|}.
   $$Hence for any bounded compact domain $\Omega$ in $O,$
   $$
  {\rm Cap}_p(\Omega,O)\geq|\Omega|\lambda_{1,p}(O).
   $$ Let $O\rightarrow M$ and $p\rightarrow\infty,$ we derive that
   $$
   \mathcal{C}(\Omega)=\limsup\limits_{p\rightarrow\infty} p\cdot\Big({\rm Cap}_p(\Omega)^{\frac 1p}\Big)\geq\lim\limits_{p\rightarrow\infty}|\Omega|^{\frac{1}{p}}
    p\Big(\lambda_{1,p}(M)\Big)^{\frac 1p}=\Lambda(M).
   $$

On the other hand, Grigor’yan [Equ.(20) in \cite{grigor1999isoperimetric}] proved that
 \begin{equation}\label{eq2.3}
m_p(O)\geq \lambda_{1,p}(O)\geq \frac{(p-1)^{p-1}}{p^{p}}m_p(O)
\end{equation}
which clearly implies that
$$
\Lambda(M)=\mathcal M.
$$
{\bf Step 2: $\mathcal V(M)\geq \mathcal C(\Omega)$.}
\par In the following, we use $S(t)$ and $V(t)$ to denote the area of $\partial B(o,t)$ and the volume $B(o,t)$ respectively. The case $\mathcal V(M)=\infty$ is trivial and if $\mathcal V(M)=0,$ then the proof is similar with the following.  Without loss of generality, we can assume that
$$
\mathcal V(M)=\limsup_{t\to \infty} \frac{\ln V(t)}{t}=1
$$
by a scaling. Then we only need to show that
$$
\limsup_{p\to \infty}p \Big({\rm Cap}_p(\Omega)\Big)^{1/p}\le 1.
$$
\par
From the hypothesis, for any $a>1$ there exists a $T>0$ such that for all $t>T$,
$$
V(t) < e^{a t}.
$$
For any bounded domain $\Omega$ in $M,$ we select a point $o\in M$ and a big $R>T$ such that $\Omega\subset B(0,R)$ and hence
$$
{\rm Cap}_p(\Omega)\le {\rm Cap}(B(0,R))\le
\Bigl(\int_R^\infty S(t)^{\frac{1}{1-p}}\,dt\Bigr)^{{1-p}}.
$$
Let us first prove that
$$
\limsup_{p\to \infty}(p-1) \Big({\rm Cap}_p(\Omega)\Big)^{\frac{1}{p-1}}\leq 1.
$$
Notice that
\begin{align*}
\limsup_{p\to \infty}(p-1) \Big({\rm Cap}_p(\Omega)\Big)^{\frac{1}{p-1}}
& \le \limsup_{p\to \infty} (p-1)\Bigl(\int_R^\infty S(t)^{\frac{1}{1-p}}\,dt\Bigr)^{-1}
\\ & =\Bigl(\liminf_{y\to 0^+}\int_R^\infty y\,S(t)^{-y}\,dt  \Bigr)^{-1}.
\end{align*}
Here we use the change of variables $y=\frac{1}{1-p}.$ Hence
we only need to show that
\begin{equation}\label{eq2.4}
  \liminf_{y\to0^+} \int_R^\infty y\,S(t)^{-y}\,dt \ge 1.
\end{equation}

Let $s=\ln V(t)$ be the smooth and increasing function and denote the inverse function by $t=F(s),$ which is also smooth and increasing. Furthermore,
$$
F'(s)=\frac{1}{(\ln V(t))'}=\frac{V(t)}{S(t)}=\frac{e^s}{S(t)}
$$
Then
\begin{equation*}
\begin{aligned}
\int_R^\infty y\,S(t)^{-y}\,dt
=& \int_{\ln V(R)}^\infty y\cdot\Big(\frac{F'(s)}{e^s}\Big)^y F'(s)\,ds\\
=&\int_{\ln V(R)}^\infty y e^{-ys}\big(F'(s)\big)^{y+1}\,ds.
\end{aligned}
\end{equation*}

Notice that
$$\int_{\ln V(R)}^\infty y\,e^{-ys}\,ds=V(R)^{-y}$$
and let $d\mu=\frac{ye^{-ys}ds}{V(R)^{-y}}$ be the probability measure.
Hence, by Jensen inequality, we have
$$
\begin{aligned}
\int_{\ln V(R)}^\infty y\,F'(s)^{1+y}e^{-ys}\,ds &=V(R)^{-y}\int_{\ln V(R)}^\infty F'(s)^{1+y}d\mu
\\&\geq V(R)^{-y}\left(\int_{\ln V(R)}^\infty F'(v)d\mu\right)^{y+1}\\
&= V(R)^{y^2}\left(\int_{\ln V(R)}^\infty y\,F'(v)e^{-yv}\,dv\right)^{y+1}.
\end{aligned}
$$
On the other hand, notice that
$$
V(t) < e^{a t}\Rightarrow F(s)\geq s/a,
$$
and hence
$$
\begin{aligned}
\int_{\ln V(R)}^\infty y\,F'(s)e^{-ys}\,ds& =ye^{-ys}F(s)|_{\ln V(R)}^{\infty}+\int_{\ln V(R)}^\infty \,y^2F(s)e^{-ys}\,ds\\
&\geq-yV(R)^{-y}F(\ln V(R))+\int_{\ln V(R)}^\infty y^2\frac{s}a e^{-ys}\,ds\\
& =O(y)+\frac1a(1+y\ln V(R))V(R)^{-y}.
\end{aligned}
$$

Let $y\to0^+$ and $a\to 1$, then \eqref{eq2.4} is proved and hence
$$
\limsup_{p \to \infty} b_p \leq 1,
$$
where $b_p = (p-1)\Big({\rm cap}_p(\Omega)\Big)^{\frac{1}{p-1}}$. In the end,
$$
a_p = p \cdot \Big({\rm cap}_p(\Omega)\Big)^{1/p} = \frac{p}{p-1} \cdot (p-1)^{\frac{1}{p}} \cdot b_p^{1-\frac{1}{p}}.
$$

For any $\varepsilon > 0$, there exists $N$ such that for all $p > N$, $b_p < 1 + \varepsilon$. Then
$$
a_p \leq \frac{p}{p-1} \cdot (p-1)^{\frac{1}{p}} \cdot (1+ \varepsilon)^{1-\frac{1}{p}}.
$$

Taking the upper limit, we obtain
$$
\limsup_{p \to \infty} a_p \leq 1 + \varepsilon.
$$

Since $\varepsilon$ is arbitrary, it follows that $\limsup\limits_{p \to \infty} a_p \leq  1.$
Therefore, we obtain the assertion (1) in Theorem 1.6.
     \item [(2)] {\bf Step 1: $\mathcal{C}(\Omega)$ is independent of $\Omega.$}
     \par If $M$ obeys the isoperimetry of $o$-centered balls, then the isoperimetric function satisfies
     $$ I(|B(o,t)|)=|\partial B(o,t)|.$$
     Let $\Omega$ be any bounded domain with $|\Omega|=|B(o,r_1)|$ for some $r_1>0,$ then
      \begin{equation}\label{eq2.5}\begin{aligned}
      {\rm Cap}_p(\Omega)&\geq\Big(\int_{|\Omega|}^{|M|} \big[I(\tau)\big]^{\frac{p}{1-p}}\,d\tau\Big)^{1-p}
        \\ &= \Big(\int_{r_1}^{\infty} \big[I(|B(o,t)|)\big]^{\frac{p}{1-p}}\frac{d |B(o,t)|}{dt}\,dt\Big)^{1-p}
        \\&=\Big(\int_{r_1}^{\infty} |\partial B(o,t)|^{\frac{1}{1-p}}\,dt\Big)^{1-p}
      \end{aligned}
      \end{equation}
      and the equality holds if $\Omega= B(o,r_1).$
      \par    On the other hand, there exists a $r_2>0$ such that $\Omega\subseteq B(o,r_2),$ and hence
    $${\rm Cap}_p(\Omega)\leq {\rm Cap}_p(B(o,r_2))\leq \Big(\int_{r_2}^{\infty} |\partial B(o,t)|^{\frac{1}{1-p}}\,dt\Big)^{1-p}$$
    Let $p\rightarrow\infty$ and we deduce that
    \begin{equation}\label{eq2.6}
      G(r_1)\leq\mathcal{C}(\Omega)\leq G(r_2)
    \end{equation}
     where
    $$
    G(r)=\limsup\limits_{p\rightarrow\infty}p\cdot\Big(\int_{r}^{\infty} |\partial B(o,t)|^{\frac{1}{1-p}}\,dt\Big)^{\frac{1-p}{p}}.$$
    \par We claim that \eqref{eq2.6} is also true if $(M,g)$ is a rotationally symmetric manifold and $\Omega$ is a bounded domain containing the center point $o.$ In fact, there exists $r_1<r_2$ such that $$B(o,r_1)\subseteq\Omega\subseteq B(o,r_2)
         $$ and hence $\forall p>1,$
         \begin{align*}
         & {\rm Cap}_p(B(o,r_1))\leq{\rm Cap}_p(\Omega)\leq{\rm Cap}_p(B(o,r_2))
         \\ &\Rightarrow \mathcal{C}(B(o,r_1))\leq\mathcal{C}(\Omega)\leq\mathcal{C}(B(o,r_2)).
         \end{align*}
         Notice that the $p-$potential of ball $B(o,r)$ is
         $$
         u(\cdot)=\frac{\int_{\rm dist_g(o,\cdot)}^\infty \varphi(s)^{\frac{n}{1-p}}\,ds}{\int_{r}^\infty \varphi(s)^{\frac{n}{1-p}}\,ds}
         $$
         where ${\rm dim}\ M=n+1$ and then
         $$
         {\rm Cap}_p(B(o,r))=\Big(\int_{r}^{\infty} \big[\omega_n\varphi(t)^n\big]^{\frac{1}{1-p}}\,dt\Big)^{1-p}=\Big(\int_{r}^{\infty} \big|\partial B(o,t)|^{\frac{1}{1-p}}\,dt\Big)^{1-p}$$
         and hence \eqref{eq2.6} still holds.
         \par We only need to prove that $G(r_1)=G(r_2).$
    Without loss of generality, we assume that $0<G(r_1)\leq G(r_2)<\infty.$ Then
    \begin{align*}
      1\leq \frac{G(r_2)}{G(r_1)}&\leq\limsup\limits_{p\rightarrow\infty}\frac{\Big(\int_{r_2}^{\infty} |\partial B(o,t)|^{\frac{1}{1-p}}\,dt\Big)^{\frac{1-p}{p}}}{\Big(\int_{r_1}^{\infty} |\partial B(o,t)|^{\frac{1}{1-p}}\,dt\Big)^{\frac{1-p}{p}}}
      \\ &=\Big(\lim\limits_{p\rightarrow\infty} \frac{\int_{r_2}^{\infty} |\partial B(o,t)|^{\frac{1}{1-p}}\,dt}{\int_{r_1}^{\infty} |\partial B(o,t)|^{\frac{1}{1-p}}\,dt}\Big)^{-1}
      \\&=1+\lim\limits_{p\rightarrow\infty}\frac{\int_{r_1}^{r_2} |\partial B(o,t)|^{\frac{1}{1-p}}\,dt}{\int_{r_2}^{\infty} |\partial B(o,t)|^{\frac{1}{1-p}}\,dt}
      \\&=1.
      \end{align*}
      Then $G(r)$ is a constant function of $r$ and hence $\mathcal{C}(\Omega)=G(r)$ is also a constant independent of $\Omega.$
      \\~\\
    \par {\bf Step 2: The equality case}
    \par If in addition, \eqref{eq1.2} holds, then$\lim\limits_{t\to \infty}\frac{|\partial B(0,t)|}{|B(o,t)|}$ exists. According to the L'Hospital's rule,
    $$\lim_{t\to \infty}\frac{\ln|B(o,t)|}{t}=\lim\limits_{t\to \infty}\frac{|\partial B(0,t)|}{|B(o,t)|}
    $$ also exists, and hence the limit is exactly
    the volume entropy $\mathcal{V}(M).$
  \par If $\mathcal{V}(M)=0,$ then \eqref{eq1.3} is trivial. In the following, we assume that $\mathcal{V}(M)>0.$
  \par In the first case where $(M,g)$ obeys the isoperimetry of $o$-centered balls, the Maz'ya constant can be estimated via \eqref{eq2.5} as
\begin{equation*}
\begin{aligned}
  m_p(M):=&\inf\limits_{F\subset\subset M}\frac{{\rm Cap}_p(F)}{|F|}\\
  =&\inf\limits_{r>0}\frac{{\rm Cap}_p(B(o,r))}{|B(o,r)|}=\inf\limits_{r>0}\frac{\Big(\int_{r}^{\infty} S(t)^{\frac{1}{1-p}}\,dt\Big)^{1-p}}{V(r)}.
\end{aligned}
\end{equation*}
  Then from \eqref{eq2.3}, we obtain
  \begin{equation}\label{eq2.7}
 \lambda_{1,p}(M)\geq \frac{(p-1)^{p-1}}{p^{p}}\cdot\inf\limits_{r>0}\frac{\Big(\int_{r}^{\infty} S(t)^{\frac{1}{1-p}}\,dt\Big)^{1-p}}{V(r)}.
\end{equation}
      \par In the second case, where $(M,g)$ is a rotationally symmetric manifold, we now show that \eqref{eq2.7} remains valid. Recall the derivation of \eqref{eq2.3} in \cite{grigor1999isoperimetric}: for any Lipschitz function $u$ on $M$ with compact support and $p>1,$
      $$
      \int_M |\nabla u|^p\,d\mu_g\geq \frac{(p-1)^{p-1}}{p^{p}}\int_0^\infty {\rm Cap}_p(U_t)\,d(t^p),
      $$ $$
      \int_M |u|^p\,d\mu_g\leq\int_0^\infty |U_t|\,d(t^p),
      $$
    where
    $U_t=\{x\in M:\ \ u(x)\geq t\}.$
    Consequently,
    \begin{equation*}
\begin{aligned}
    \frac{\int_M |\nabla u|^p\,d\mu_g}{\int_M |u|^p\,d\mu_g}\geq& \frac{(p-1)^{p-1}}{p^{p}}\frac{\int_0^\infty {\rm Cap}_p(U_t)\,d(t^p)}{\int_M |u|^p\,d\mu_g}\\
    \geq& \frac{(p-1)^{p-1}}{p^{p}} \inf\limits_{F\subset\subset M}\frac{{\rm Cap}_p(F)}{|F|}.
\end{aligned}
\end{equation*}
    Now, on a rotationally symmetric manifold, we may choose a radial function $u$ as a test function to estimate the first $p-$eigenvalue. By restricting to functions depending only on the distance $r$ from the pole $o,$ the level sets $U_t$ become concentric $o-$centered balls. It implies that the infimum over all compact subsets $F$ in the estimate above can be replaced by the infimum over $o-$centered balls. Hence the same inequality \eqref{eq2.7} holds in the rotationally symmetric setting as well.
    \par We now set
    $$f(r):=\frac{\Big(\int_{r}^{\infty} S(t)^{\frac{1}{1-p}}\,dt\Big)^{1-p}}{V(r)},\ \ r>0
    $$ and try to find the infimum of $f(r).$ A direct calculation indicates that
    \begin{align*}
    \frac{d}{dr} \big(f(r)\big)^{\frac{1}{1-p}}&= V(r)^{\frac{1}{p-1}-1}S(r)\left(\frac{1}{p-1}\int_r^\infty S(t)^{-\frac{1}{p-1}}\,dt-V(r)S(r)^{-\frac{1}{p-1}-1}\right)
    \\&:=V(r)^{\frac{1}{p-1}-1}S(r)g(r)
    \end{align*}
    Recall that
   \begin{equation}\label{eq2.8}
     \lim\limits_{r\to\infty}\frac{\ln V(r)}{r}=\lim\limits_{r\to\infty}\frac{S(r)}{V(r)}=\mathcal{V}(M)>0,
   \end{equation}
    then $V(r)\sim S(r)\sim e^{\mathcal{V}(M)}$ as $r\rightarrow\infty$ and hence $g(\infty)=\lim\limits_{r\to\infty}g(r)=0.$ Moreover,
    \begin{align*}
       g'(r)&= -\frac{1}{p-1}  S(r)^{-\frac{1}{p-1}}-S(r)^{-\frac{1}{p-1}}+\Big(\frac{1}{p-1}+1\Big)V(r)S(r)^{-\frac{1}{p-1}-2}S'(r)\\
        & =\Big(\frac{1}{p-1}+1\Big)S(r)^{-\frac{1}{p-1}-2}\left(V(r) S'(r)-S(r)^2\right) \\
        & =\Big(\frac{1}{p-1}+1\Big)S(r)^{-\frac{1}{p-1}-2}V(r)^2\cdot\frac{d}{dr}\left(\frac{S(r)}{V(r)}\right)
    \end{align*}
    By \eqref{eq1.2}, there exists a $R_0>0$ such that $\frac{S(r)}{V(r)}$ is decreasing on $[R_0,\infty).$ Then $g(r)$ is also decreasing on $[R_0,\infty)$ and hence $g\geq g(\infty)=0$ on $[R_0,\infty).$ As a consequence, $f(r)$ is decreasing on $[R_0,\infty).$ Notice that
    $$
    \lim_{r\to\infty}f(r)= \lim_{r\to\infty}\left(\frac{\int_{r}^{\infty} S(t)^{\frac{1}{1-p}}\,dt}{V(r)^{\frac{1}{1-p}}}\right)^{1-p}
    =\frac{\mathcal{V}(M)^p}{(p-1)^{p-1}}.
    $$
    and
    $$
    f(r)\geq \frac{\Big(\int_{0}^{\infty} S(t)^{\frac{1}{1-p}}\,dt\Big)^{1-p}}{V(R_0)}:=c_p,\ \ \ \ \forall r\in (0,R_0].
    $$
    Then
    \begin{equation}\label{eq2.9}
   \inf_{r>0} f(r)\geq\min\{c_p,\ \frac{\mathcal{V}(M)^p}{(p-1)^{p-1}}\}.
    \end{equation}

    Combining \eqref{eq2.7} and \eqref{eq2.9}, we have
    \begin{equation}\label{eq2.10}
    \begin{aligned}
      \Lambda(M)&=\lim_{p\to\infty}p\cdot\Big(\lambda_{1,p}(M)\Big)^{\frac1p} \\
      &\geq \limsup_{p\to\infty}p\Big(\frac{(p-1)^{p-1}}{p^{p}}\Big)^{\frac1p}\cdot\min\bigg\{(c_p)^{\frac1p},\ \Big(\frac{\mathcal{V}(M)^p}{(p-1)^{p-1}}\Big)^{\frac1p}\bigg\}
     \\ & =\min\{\limsup_{p\to\infty}(p-1)(c_p)^{\frac1p},\ \mathcal{V}(M)\}.
    \end{aligned}
    \end{equation}
    We can assume that $c_p<1$ (otherwise $\limsup\limits_{p\to\infty}(p-1)(c_p)^{\frac1p}=\infty$) and then
\begin{equation*}
\begin{aligned}
    \limsup_{p\to\infty}(p-1)(c_p)^{\frac1p} \geq& \limsup_{p\to\infty}(p-1)(c_p)^{\frac1{p-1}}\\
    =&\limsup_{p\to\infty}(p-1)\Big(\int_{0}^{\infty} S(t)^{\frac{1}{1-p}}\,dt\Big)^{-1}.
\end{aligned}
\end{equation*}
   From \eqref{eq2.8}, for any $\varepsilon>0,$ there exists a $T_0>R_0>0,$ such that
   $$
   \forall t>T_0,\ \ S(t)\geq \mathcal{V}(M) \cdot V(t)\geq \mathcal{V}(M) \cdot e^{(\mathcal{V}(M) -\varepsilon)t}.
   $$
  Then
   \begin{align*}
   \limsup_{p\to\infty}(p-1)(c_p)^{\frac1p}&\geq\Big(\liminf_{p\to\infty}\frac{1}{p-1}\int_{0}^{\infty} S(t)^{\frac{1}{1-p}}\,dt\Big)^{-1}
   \\ &=\left(\liminf_{p\to\infty}\frac{1}{p-1}\Big(\int_{0}^{T_0} S(t)^{\frac{1}{1-p}}\,dt+\int_{T_0}^{\infty} S(t)^{\frac{1}{1-p}}\,dt\Big)\right)^{-1}
   \\& \geq \left(\liminf_{p\to\infty}\frac{1}{p-1}\int_{T_0}^{\infty}\Big(\mathcal{V}(M) \cdot e^{(\mathcal{V}(M) -\varepsilon)t}\Big)^{\frac{1}{1-p}}\,dt\right)^{-1}
   \\&=\mathcal{V}(M) -\varepsilon
   \end{align*}
   Let $\varepsilon\rightarrow 0$ and \eqref{eq2.10} implies that
   $$ \Lambda(M)\geq \mathcal{V}(M).$$
   In the end, we obtain \eqref{eq1.3} due to \eqref{eq1.1}.
  \end{itemize}
     \end{proof}

     \begin{proof}[Proof of Theorem \ref{thm1.9}]
      \begin{itemize}
     \item [(1)] If ${\rm Ric}\geq 0,$ then by the Bishop-Gromov theorem, $$
     |\partial B(o,t)|\leq \omega_n t^n,\ \ \forall o\in M,\ t>0.$$
     Then ${\rm Cap}_p(B(o,r))=0$ for all $r>0$ and $p\geq n+1.$ Namely, $M$ is $p-$parabolic and hence $\mathcal{C}(\Omega)\equiv 0.$
   \item [(2)] By the Bishop--Gromov volume comparison theorem, for any complete Riemannian manifold of dimension $n+1$ with $\operatorname{Ric}\ge -n,$ the function
$$
r\;\longmapsto\; \frac{\operatorname{Vol}(B_o(r))}{V_\mathbb{H}(r)}
$$
is nonincreasing, where $V_\mathbb{H}(r)$ denotes the volume of a ball of radius $r$ in hyperbolic space.
 So there exists a constant $C>0$ such that
$$
\operatorname{Vol}(B_o(r)) \le C\,V_\mathbb{H}(r) \qquad\forall\,r>0.
$$
It is well-known that
$$
V_\mathbb{H}(r) \sim \frac{\omega_n}{2^n}\,e^{nr}\quad\text{as } r\to\infty,
$$
where $\omega_n$ is the volume of the unit $n$-sphere. Consequently,
$$
\limsup_{r\to\infty}\frac{\ln\operatorname{Vol}(B_o(r))}{r}
\le \limsup_{r\to\infty}\frac{\ln\bigl(C\,V_\mathbb{H}(r)\bigr)}{r}=n.
$$
Thus $\mathcal V(M)\le n$, completing the proof due to \eqref{eq1.1}.
     \item [(3)] According to Theorem \ref{thm1.6} (1) and Theorem \ref{thm1.9} (2), we have that
     $$
     \Lambda(M)\leq\mathcal{C}(\Omega)\le \mathcal V(M)\leq n,\ \ \forall\Omega\subseteq M.
     $$
     Recall the estimates of $p-$eigenvalue in \cite [Corollary 1.3]{jin2025lower}), under the conditions of (3), we obtain
     $\lambda_{1,p}(\Omega_k)\geq\big(\frac{n}{p}\big)^p$ for any $\Omega_k$ and $p>1.$ Let $\Omega_k\rightarrow M$ and $p\rightarrow\infty$ and then
     $\Lambda(M)=n.$
     \item [(4)]We will first introduce some basic materials  regarding conformally compact Einstein (also known as Poincar\'e Einstein) manifolds. Suppose that $\overline{M}$ is $n+1-$dimensional manifold with smooth boundary $\partial M$ of dimension $n$ and $M$ is its interior. A complete noncompact metric $g$ in $M$ is called conformally compact if there exists a defining function $\rho$ such that
         $$
         \rho>0\ \ \text{in} \ \  M, \ \ \ \ \rho=0\ \ \text{on} \ \partial M,\ \ \ \ \ d\rho\neq 0 \ \ \text{on} \  \ \partial M
         $$ and the conformal metric $\bar{g}=\rho^2g$ extends to a $C^2$ metric on $\overline{M}.$ The restriction $\hat{g}=\bar{g}|_{T\partial M}$ defines a boundary metric associated with the compactification $\bar{g}.$  Consequently, $(M,g)$ induces a conformal structure $(\partial M,[\hat{g}])$ on the boundary, called the conformal infinity of $(M,g).$
         If, in addition, $g$ is Einstein, i.e. ${\rm Ric}[g]=-ng,$ then $(M,g)$ is a conformally compact Einstein manifold. A direct computation shows that $|d\rho|^2_{\rho^2g}|_{\partial M}=1,$ and $K[g]=-1+O(\rho^2).$ Hence $(M,g)$ is asymptotically hyperbolic and serves as a natural extension of hyperbolic space.
         \par A classical result due to Lee \cite{lee1995spectrum} establishes a fundamental connection between the Yamabe constant of the conformal infinity and the spectrum of the interior Laplacian on a conformally compact Einstein manifold, i.e.
         $$
         Y(\partial M,[\hat{g}])\geq 0\Rightarrow \lambda_{1,2}(M)=\frac{n^2}{4}.
         $$
         The result was later extended to the general $p-$case in \cite{hijazi2020cheeger} and \cite{jin2025lower}:
         $$
         Y(\partial M,[\hat{g}])\geq 0\Rightarrow \lambda_{1,p}(M)=\frac{n^p}{p^p},\ \ \forall p>1.
         $$ Hence we get that $\Lambda(M)=n$ when $Y(\partial M,[\hat{g}])\geq 0$ and finish the proof.
   \end{itemize}
 \end{proof}

\section {Relation to topological entropy}

\begin{proof}[Proof of Theorem \ref{thm:1.10}.]

Let $M$ be the universal cover of $N$ and
$$S(t)=|\partial B(o,t)|\ \ \&\ \ V(t)=|B(o,t)|.$$
  Knieper \cite{Knieper1997} proved that there exist constants $a>1$, $r_0>0$ and $\alpha=({\rm rank} M-1)/2\ge 0$ such that for all $t\ge r_0$,
\begin{equation}\label{eq3.1}
\frac{1}{a}\,t^\alpha e^{\mathcal{V}(M)t}\le S(t)\le a\,t^\alpha e^{\mathcal{V}(M)t}.
\end{equation}
If $M$ is flat, the assertion is obvious and are equal to zero.  Next, we consider that $M$ is non-flat and assume that $\mathcal{V}(M)>0.$
Under either hypothesis (A) or (B), according to \eqref{eq2.7} in Section 2,
\begin{equation}\label{eq3.2}
\lambda_{1,p}(M)\ge \frac{(p-1)^{p-1}}{p^{p}}\;\inf_{r>0}\frac{{\rm Cap}_p(B(o,r))}{V(r)}.
\end{equation}

If $r\geq r_0,$ then by the lower bound for $S(t)$ in \eqref{eq3.1}, we obtain

\begin{equation}\label{eq3.3}
\begin{aligned}
{\rm Cap}_p(B(o,r))& = \left(\int_r^\infty S(t)^{\frac{1}{1-p}}dt\right)^{1-p}
\\& \ge \left(\int_r^\infty\Big(\frac{1}{a}\,t^\alpha e^{\mathcal{V}(M)t}\Big)^{\frac{1}{1-p}}\,dt\right)^{1-p}
\\ & \geq \left(\int_r^\infty\Big(\frac{1}{a}\,r^\alpha e^{\mathcal{V}(M)t}\Big)^{\frac{1}{1-p}}\,dt\right)^{1-p}
\\ & =\frac{r^\alpha}{a}\Big(\frac{p-1}{\mathcal{V}(M)}\Big)^{1-p}e^{\mathcal{V}(M)r}.
\end{aligned}
\end{equation}
On the other hand
\begin{equation}\label{eq3.4}
V(r)\le V(r_0)+\int_{r_0}^ra\,t^\alpha e^{\mathcal{V}(M)t}\,dt\leq Cr^\alpha e^{\mathcal{V}(M)r}
\end{equation}
where $C$ only depends on $a,r_0, V(r_0)$ and $\mathcal{V}(M).$
Combining \eqref{eq3.3} and \eqref{eq3.4} gives
$$
\frac{{\rm Cap}_p(B(o,r))}{V(r)}\ge \frac{1}{aC}\Big(\frac{p-1}{\mathcal{V}(M)}\Big)^{1-p}$$
The right-hand side is independent of $r$; therefore the infimum over $r$ satisfies the same estimate:
\begin{equation}\label{eq3.5}
\inf_{r>0}\frac{{\rm Cap}_p(B(o,r))}{V(r)}\ge \min\{c_p,\ \frac{1}{aC}\Big(\frac{p-1}{\mathcal{V}(M)}\Big)^{1-p}\},
\end{equation}
where
$$
   c_p:= \frac{\Big(\int_{0}^{\infty} S(t)^{\frac{1}{1-p}}\,dt\Big)^{1-p}}{V(r_0)}
    $$

Plugging \eqref{eq3.5} into \eqref{eq3.2} yields
\begin{equation}
    \begin{aligned}
      \Lambda(M)&=\lim_{p\to\infty}p\cdot\Big(\lambda_{1,p}(M)\Big)^{\frac1p} \\
      &\geq \limsup_{p\to\infty}p\Big(\frac{(p-1)^{p-1}}{p^{p}}\Big)^{\frac1p}\cdot\min\bigg\{(c_p)^{\frac1p},\ \Big(\frac{1}{aC}\Big(\frac{p-1}{\mathcal{V}(M)}\Big)^{1-p}\Big)^{\frac1p}\bigg\}
     \\ & =\min\{\limsup_{p\to\infty}(p-1)(c_p)^{\frac1p},\ \mathcal{V}(M)\}.
    \end{aligned}
    \end{equation}
Similar to the proof in Section 2, we obtain that
   $$ \Lambda(M)\geq \mathcal{V}(M).$$
  From Theorem 1.6(1), we have the general inequalities
$$
 \mathcal V(M)=\mathcal{C}(\Omega)= \Lambda(M)= \mathcal M(M).
$$

In the end, Manning \cite{Ma1979} proved that if the sectional curvatures of $N$ is non-positive, then $h_{\rm top}(\phi)=\mathcal V(M)$.

Thus we prove the assertion.

\end{proof}

\begin{proof}[Proof of Corollary \ref{thm:1.11}.]
Suppose that $M$ is universal cover of $N$, and the projection $\pi: M \to N$ is a local isometry, so the Ricci curvature of $M$ satisfies the same lower bound:
$$
\mathrm{Ric}_M \ge -n\,g_M.
$$

It follows from Theorem 1.9 that $\mathcal{V}(M)\leq n.$

On the other hand, since $\lambda_{1,p}\geq\Big(\frac{n}p\Big)^p$, we have  $\mathcal{V}(M)\geq\Lambda(M)\geq n$.
Therefore $\mathcal{V}(M)=n$.

The Ledrappier–Wang theorem (compact version) states \cite{LW2010}:
\begin{quote}
Let $N$ be an $m$-dimensional closed Riemannian manifold with \(\mathrm{Ric}_N \ge -(m-1)g_N\). If its volume entropy is $m-1$, then $N$ is isometric to a compact hyperbolic manifold (constant curvature $-1.$)
\end{quote}
Therefore \(N\) is isometric to a compact hyperbolic manifold.
Then, the universal cover of a compact hyperbolic manifold is the hyperbolic space \(\mathbb{H}^{n+1}\). Consequently \(M \cong \mathbb{H}^{n+1}\). This completes the proof.
\end{proof}

```latex
\section{Proof of the second-order refinement}
\label{sec:proof-second-order-refinement}

We prove Theorem~\ref{thm:second-order-refinement-110}.  The proof uses
two ingredients.  The first one is the finite-$p$ comparison already used
in the proof of Theorem~\ref{thm1.6}.  The second one is a simple
Laplace-type estimate.

Let
\[
        S(r)=|\partial B(o,r)|,
        \qquad
        I_p(R)=\int_R^\infty S(t)^{\frac{1}{1-p}}\,dt.
\]

\begin{lemma}
\label{lem:finite-p-squeezing}
Assume that $(M,g,\Omega,o)$ satisfies one of the alternatives
{\rm (A)} and {\rm (B)} in Theorem~\ref{thm1.6}(2).  Then there exist
radii $R_-,R_+>0$, independent of $p$, such that
\[
        \frac{p-1}{I_p(R_-)}
        \leq
        (p-1){\rm Cap}_p(\Omega)^{\frac{1}{p-1}}
        \leq
        \frac{p-1}{I_p(R_+)}.
\]
\end{lemma}

\begin{proof}
This is precisely the finite-$p$ comparison appearing in the proof of
Theorem~\ref{thm1.6} as
$$
        {\rm Cap}_p(B(o,R))=I_p(R)^{1-p}.
$$
and there are $R_-,R_+>0$ such that
$$
       {\rm Cap}_p(B(o,R_-))
        \leq
        {\rm Cap}_p(\Omega)
        \leq
       {\rm Cap}_p(B(o,R_+)).
$$
\end{proof}

\begin{lemma}
\label{lem:laplace-estimate}
Assume that, for some $H>0$, $\alpha\geq0$, $C>1$, and $R_0>0$,
\[
        C^{-1}r^\alpha e^{Hr}
        \leq
        S(r)
        \leq
        Cr^\alpha e^{Hr},
        \qquad r\geq R_0.
\]
Then, for every fixed $R>0$,
\[
        \frac{p-1}{H I_p(R)}
        =
        1+\frac{\alpha\ln p}{p}
        +O\!\left(\frac1p\right).
\]
\end{lemma}

\begin{proof}
Put $y=(p-1)^{-1}\rightarrow 0$ as $p\rightarrow\infty.$  We need the asymptotics of
\[
        I_y(R):=\int_R^\infty S(t)^{-y}\,dt.
\]
For the model integral
$$
        J_y(R):=\int_R^\infty (t^{\alpha }e^{Ht})^{-y}\,dt =
        (Hy)^{\alpha y-1}
        \int_{HyR}^{\infty}s^{-\alpha y}e^{-s}\,ds.
$$

Since
\[
        \int_{HyR}^{\infty}s^{-\alpha y}e^{-s}\,ds=1+O(y),
\]
we obtain
\begin{equation}\label{eq4.1}
  \begin{aligned}
  \frac{1}{Hy J_y(R)}&=e^{-\alpha y\ln (Hy)}\big(1+O(y)\big)^{-1}
\\ &=(1-\alpha y\ln (Hy)+O((y\ln y)^2))\big(1+O(y)\big)
\\ &=1-\alpha y\ln y+O(y).
  \end{aligned}
\end{equation}

Now choose $T\geq\max\{R,R_0,1\}$ and hence
$$
I_y(R)=\int_R^T S(t)^{-y}\,dt+I_y(T)=T-R+O(y)+I_y(T)
$$
where
$$
        C^{-y}J_y(T)
        \leq
        I_y(T)
        \leq
        C^yJ_y(T)
$$
would imply that
$$
\frac{1}{Hy I_y(T)}=1-\alpha y\ln y+O(y).
$$
Therefore, we obtain that
$$
\frac{1}{Hy I_y(R)}=1-\alpha y\ln y+O(y).
$$
Then we finish the proof by setting $y=\frac{1}{p-1}.$
\end{proof}

\begin{proof}[Proof of Theorem~\ref{thm:second-order-refinement-110}]
By Theorem~\ref{thm:1.10}, in the compact-quotient setting the volume
entropy is the genuine limit
\[
        \mathcal V(M)=\lim_{r\to\infty}\frac{\ln |B(o,r)|}{r},
\]
and
\[
     \mathcal   V(M)=\mathcal C(\Omega)=\Lambda(M)=\mathcal M(M).
\]

Let $k=\operatorname{rank}M$.  Since $M$ is a nonflat Hadamard manifold
with nonpositive sectional curvature and compact quotient, Knieper's
sphere-growth theorem gives
\[
        S(r)\sim
        r^{(k-1)/2}e^{\mathcal V(M)r}.
\]
Thus Lemma~\ref{lem:laplace-estimate} applies with
\[
        H=\mathcal V(M),
        \qquad
        \alpha=\frac{k-1}{2}.
\]
Hence, for every fixed $R>0$,
\[
        \frac{p-1}{\mathcal V(M) I_p(R)}
        =1+ \frac{k-1}{2}\frac{\ln p}{p}
        +
        O\!\left(\frac1p\right).
\]

By Lemma~\ref{lem:finite-p-squeezing},
\[
        \frac{p-1}{I_p(R_-)}
        \leq
        (p-1){\rm Cap}_p(\Omega)^{\frac{1}{p-1}}
        \leq
        \frac{p-1}{I_p(R_+)}
\]
for some fixed radii $R_-,R_+>0$.  Since both endpoints have the same
second-order expansion, we obtain
\[
        \frac{(p-1){\rm Cap}_p(\Omega)^{\frac{1}{p-1}}}{\mathcal V(M)}
        =1        +\frac{k-1}{2}
        \frac{\ln p}{p}
        +
        O\!\left(\frac1p\right).
\]
Therefore
\[
        \operatorname{rank}M
        =
        1+2\alpha_{\operatorname{cap}}(\Omega).
\]
\end{proof}

\begin{proof}[Proof of Corollary~\ref{cor:second-order-consequences}]
Set
$$
        \Theta_p(\Omega)
        :=
        (p-1){\rm Cap}_p(\Omega)^{\frac{1}{p-1}}=
        \mathcal V(M)
        +
        \mathcal V(M) \left(\frac{\operatorname{rank}M-1}{2}\right)\frac{\ln p}{p}
        +
        O\!\left(\frac1p\right).
$$
Then
we have
\[
        p{\rm Cap}_p(\Omega)^{1/p}
        =
        p
        \left(
        \frac{\Theta_p(\Omega)}{p-1}
        \right)^{(p-1)/p}.
\]
Taking logarithms gives
$$
\begin{aligned}
        \ln\left(
        p{\rm Cap}_p(\Omega)^{\frac1p}
        \right)
        &=  \ln p +   \left(1-\frac1p\right) \left( \ln\Theta_p(\Omega)-\ln(p-1)\right)\\
        &=  \ln\Theta_p(\Omega)+\ln p-\ln(p-1)+\frac{\ln(p-1)}{p}-\frac{\ln\Theta_p(\Omega)}{p}\\
        &=\ln\Theta_p(\Omega)+\frac{\ln p}{p}+O\!\left(\frac1{p}\right).
\end{aligned}
$$

Exponentiating yields
$$
        p{\rm Cap}_p(\Omega)^{\frac1p}
        =\Theta_p(\Omega)e^{\frac{\ln p}{p}+O\!\left(\frac1{p}\right)}=\Theta_p(\Omega)\Big(1+\frac{\ln p}{p}+O\!\left(\frac1{p}\right)\Big).
$$
This becomes
$$
        p{\rm Cap}_p(\Omega)^{\frac1p}=\mathcal V(M) + \mathcal V(M)
        \left(
        \frac{\operatorname{rank}M+1}{2}
        \right)\frac{\ln p}{p}
        +
        O\!\left(\frac1p\right).
$$
Hence \eqref{eq1.8} holds.

\end{proof}

\section{Some examples}
 \begin{example}\label{ex3.1}
        In order to construct a rotationally symmetric manifold $$(M,g)=([0,\infty),\ dt^2+\varphi(t)^2g_{\mathbb{S}^n})$$ with center $o=\{t=0\}$ and the area of the geodesic sphere is $|\partial B(o,t)|=\omega_n \varphi(t)^n,$ we only construct a smooth, strictly increasing function $g(t)=\ln |\partial B(o,t)|$ as follows.
        \\ Define a sequence $\{x_n\}$ by
         $$
        x_0=1,\ \ x_1=2,\ \ x_{n+1}=x_n+1+e^{x_n+1},\ \forall n\geq 1.
        $$
        Define $g$ on each interval $[x_n,x_{n+1}]$ in two parts.
        \begin{itemize}
          \item On $[x_n+1,x_{n+1}],$ set $g(t)=x_n+1.$
          \item On $[x_n,x_n+1],$ let $g$ be smooth and strictly increasing, and require that all derivatives at $x_n$ and $x_n+1$ vanish.
        \end{itemize}
        Thus $(M,g)$ is a smooth manifold with $|\partial B(o,t)|$ increasing.
        \par On the one hand,
        $$ |B(o,x_n+2)|\geq\int_{x_n+1}^{x_n+2} |\partial B(o,t)|\,dt=e^{x_n+1}
        $$ and hence
        $$
        \mathcal{V}(M)=\limsup_{R\rightarrow\infty}\frac{\ln |B(o,R)|}{R}\geq \limsup_{n\rightarrow\infty}\frac{\ln |B(o,x_n+2)|}{x_n+2}=1.
        $$
        One can even check that $\mathcal{V}(M)=1.$ On the other hand, for any $p>2,$
        \begin{align*}
            {\rm Cap}_p(B(o,1))&= \Big(\int_{1}^{\infty} \big|\partial B(o,t)|^{\frac{1}{1-p}}\,dt\Big)^{1-p}
            \\ &\leq\Big(\sum_{n=1}^\infty\int_{x_n+1}^{x_{n+1}} e^{\frac{x_n+1}{1-p}}\,dt \Big)^{1-p}
            \\ &=\Big(\sum_{n=1}^\infty  e^{(1+x_n)(1+\frac{1}{1-p})} \Big)^{1-p}
            \\ &=0.
        \end{align*}
        As a consequence, $\mathcal{C}(B(o,1))=0$ and $\mathcal{C}(B(o,1))<\mathcal{V}(M).$
     \end{example}

     \begin{example}\label{ex3.2}
       Let $(M,g)$ be a rotationally symmetric manifold with $g= dt^2+\varphi(t)^2g_{\mathbb{S}^n}$ and $|\partial B(o,t)| = e^{t h(t)}$, where $h(t) = 4 + \sin\ln t$ for $t\geq 1.$ We modify $\varphi(t)$ near $t=0$ so that the metric extends smoothly. Then
       \begin{enumerate}
\item $|\partial B(o,t)|$ is strictly increasing in $t.$
\item The limit $\displaystyle\lim_{p\to\infty} p\Big({\rm Cap}_p(B(o,1))\Big)^{\frac1p}$ does not exist.
\item $\mathcal{C}(B(o,1))>\Lambda(M).$
\end{enumerate}
\begin{proof}
\begin{enumerate}
  \item Since $$
  \frac{d}{dt}|\partial B(o,t)|= e^{t h(t)}(4+\sin\ln t+\cos\ln t)>0,
  $$ We get that $|\partial B(o,t)|$ is a strictly increasing function.
  \item  Since
  \begin{equation}\label{eq4.1}
  \displaystyle\lim_{p\to\infty} p\Big({\rm Cap}_p(B(o,1))\Big)^{\frac1p}=\displaystyle\lim_{p\to\infty} (p-1)\Big({\rm Cap}_p(B(o,1))\Big)^{\frac1{p-1}},
  \end{equation} we only consider the limit on the right side. Let
  $$x = \frac{1}{p-1}\rightarrow 0^+,\ \ \&\ \ \theta = x t.$$
  Then
\begin{align*}
(p-1)\Big({\rm Cap}_p(B(o,1))\Big)^{\frac1{p-1}} &= (p-1)\left( \int_1^\infty e^{t \cdot h(t)/(1-p)}\,dt\right)^{-1}
\\ &=\left( \int_1^\infty xe^{-x t \cdot h(t)}\,dt\right)^{-1}
\\&= \left( \int_x^\infty e^{-\theta h(\theta/x)}\,d\theta \right)^{-1}
\end{align*}
Hence we only need to check that whether the limit
$$
\lim_{x\to0^+}\int_x^\infty e^{-\theta h(\frac{\theta}{x})}\,d\theta =\lim_{x\to0} \int_x^\infty e^{-\theta\left(4+\sin\ln\frac{\theta}{x}\right)}\,d\theta
$$ exists or not.
Let
$$
x_k = e^{-2k\pi}, \qquad x_k' = e^{-2k\pi + \pi}, \qquad k\in\mathbb{N}.
$$
Clearly $x_k, x_k' \to 0$ as $k\to\infty$  and
$$\begin{cases}
\lim\limits_{k\to\infty}\int_{x_k}^\infty e^{-\theta h(\frac{\theta}{x_k})}\,d\theta = \int_0^\infty e^{-\theta(4+\sin\ln\theta)}\,d\theta \equiv I_1
\\ \lim\limits_{k\to\infty}\int_{x'_k}^\infty e^{-\theta h(\frac{\theta}{x'_k})}\,d\theta = \int_0^\infty e^{-\theta(4-\sin\ln\theta)}\,d\theta \equiv I_2.
\end{cases}
$$
In the following, we will show that
$$
I_2-I_1=2\int_0^\infty e^{-4\theta}\sinh(\theta\sin\ln\theta)\,d\theta<0
$$ and it would imply that the limit of \eqref{eq3.1} does not exist.
\par
Using $\sinh y = \sum\limits_{k=0}^\infty \frac{y^{2k+1}}{(2k+1)!}$ and integrating termwise,
$$
I_2-I_1 = 2\sum_{k=0}^\infty \frac{\int_0^\infty e^{-4\theta} \theta^{2k+1} (\sin \ln \theta)^{2k+1}\,d\theta}{(2k+1)!}.
$$ On the one hand,
\begin{align*}
\sum_{k=1}^\infty\frac{\int_0^\infty e^{-4\theta} \theta^{2k+1} (\sin \ln \theta)^{2k+1}\,d\theta}{(2k+1)!}
& \le \sum_{k=1}^\infty \frac{\int_0^\infty e^{-4\theta} \theta^{2k+1} d\theta }{(2k+1)!}
\\ &=\sum_{k=1}^\infty \frac{1}{4^{2k+2}}=\frac{1}{240}.
\end{align*}
On the other hand,
$$
\int_0^\infty e^{-4\theta} \theta \sin(\ln \theta)\,d\theta=\int_{-\infty}^{\infty} e^{2t - 4e^t} \sin t \, dt.
$$
Split the integral:
$$
 \int_{-\infty}^{-\pi} + \int_{-\pi}^{\pi} + \int_{\pi}^{\infty} =: A + B + C
$$ where
$$
A \le \int_{-\infty}^{-\pi} e^{2t - 4e^t}\,dt
      = \int_0^{e^{-\pi}} u e^{-4u}\,du
      = \frac{1-e^{-4e^{-\pi}}(4e^{-\pi}+1)}{16}<0.0009
$$
and
$$
C \le \int_{\pi}^{\infty} e^{2t - 4e^t}\,dt
      = \int_{e^{\pi}}^{\infty} u e^{-4u}\,du
      =\frac{e^{-4e^\pi}(4e^\pi+1)}{16}<10^{-38}.
$$
On $[0,\pi]$, $e^t \ge 1+t$ and convexity gives
 $e^{-t} \le  1 - a t$ with $a = (1-e^{-\pi})/\pi > 0.$ Then
\begin{align*}
  B &=   \int_{-\pi}^{\pi} e^{2t - 4e^t} \sin t \, dt
 \\& =  \int_0^{\pi} \bigl( e^{2t - 4e^t} - e^{-2t - 4e^{-t}} \bigr) \sin t \, dt
 \\ &\leq  \int_0^{\pi} e^{-4}\bigl(e^{-2t} - e^{-(2-4a)t}\bigr) \sin t \, dt
 \\&=e^{-4}\Big(\frac{1+e^{-2\pi}}{5} - \frac{1+e^{-(2-4a)\pi}}{(2-4a)^2+1}\Big)
 \\&< -0.008.
\end{align*}
As a consequence,
$$I_2-I_1 \leq A+B+C+\frac{1}{240}<0
$$ and we finish the proof.
  \item Recall the property of Maz'ya constant implies that
  $$ {\rm Cap}_p(\Omega)\geq|\Omega|\lambda_{1,p}(M).$$
  where $\Omega=B(o,1)$ as above. Then
  \begin{align*}
  \mathcal{C}(\Omega)&=\limsup\limits_{p\rightarrow\infty}p\cdot\Big({\mathrm{Cap}_p(\Omega)}\Big)^{\frac 1p}
  \\ & >\liminf\limits_{p\rightarrow\infty}p\cdot\Big({\mathrm{Cap}_p(\Omega)}\Big)^{\frac 1p}
  \\ &\geq\liminf\limits_{p\rightarrow\infty}p\cdot \Big(|\Omega|\lambda_{1,p}(M)\Big)^{\frac 1p}
  \\ &=\Lambda(M).
  \end{align*}
\end{enumerate}
\end{proof}
     \end{example}

\begin{example}[Product structure]
\label{ex:product-structure}
Let
\[
        M=M_1\times\cdots\times M_m
\]
be a Riemannian product of $m$ rank-one Hadamard manifolds, each admitting
a compact quotient.  Then $\operatorname{rank}M=m$, and the radial
second-order quantity
\[
        \Theta^{\operatorname{rad}}_p(R)
        :=
        \frac{p-1}{
        \displaystyle
        \int_R^\infty |\partial B(o,t)|^{\frac{1}{1-p}}\,dt
        }
\]
satisfies
\[
        \frac{\Theta^{\operatorname{rad}}_p(R)}{\mathcal V(M)}
        =1+\frac{m-1}{2} \frac{\ln p}{p}+O\!\left(\frac1p\right).
\]
Hence
\[
        \lim_{p\to\infty}
        \frac{
        p\left(\Theta^{\operatorname{rad}}_p(R)-\mathcal V(M)\right)
        }{
        \mathcal V(M)\ln p
        }
        =
        \frac{m-1}{2}.
\]
Thus the radial second-order exponent detects the number of rank-one
product factors.  If the true $p$-capacity is uniformly comparable to the
radial capacity as $p\to\infty$, the same limit holds for the true
capacity.
\end{example}

\section{Appendix: The infinity capacity of a condenser}\label{Appendixsection}
 Let $(K,O)$ be a condenser in $(M,g),$ we define the infinity capacity by
\begin{equation*}
\begin{aligned}
 {\rm Cap}_{\infty}(K,O)=&\lim\limits_{p\rightarrow\infty} \Big({\rm Cap}_{p}(K,O)\Big)^{\frac 1p}\\
 =&\inf\left\{\|\nabla u\|_{L^\infty(O\setminus K)} : u\in W^{1,\infty}_0(O),\; u\ge 1\text{ on }K\right\}.
\end{aligned}
\end{equation*}
It is shown in \cite{jiang2018relative} that the limit always exists and in Euclidean space $\mathbb{R}^n,$ there holds
$$
{\rm Cap}_{\infty}(K,O)=\frac{1}{{\rm dist}(\partial K,\partial O)}.
$$
Following the ideal of \cite{jiang2018relative}, we have
\begin{theorem}\label{A1}
  Assume that $(M,g)$ is a complete Riemannian manifold of dimension $n+1$. Let $(K,O)$ be a condenser with smooth boundaries. Then
  $$
  {\rm Cap}_{\infty}(K,O)=\frac 1{d_g(\partial K,\partial O)},
  $$
  where $d_g(\partial K,\partial O)=\inf\{d_g(x,y):x\in\partial K,\;y\in\partial O\}$ is the Riemannian distance between $\partial K$ and $\partial O$.
\end{theorem}
\begin{proof}
The proof follows the strategy of Jiang--Xiao--Yang \cite{jiang2018relative}. The only points
which require verification are the compactness argument and Jensen's
Lipschitz identity in the Riemannian setting. The former follows from
the Sobolev--Morrey embedding and weak lower semicontinuity on bounded
smooth Riemannian domains. The latter follows from the equivalence
between infinity harmonic functions and AMLEs on Riemannian manifolds.
\end{proof}

\begin{theorem}[Riemannian version of Corollary~1.13 in \cite{jensen1993uniqueness}]\label{A1}
    Let $(M,g)$ be a smooth Riemannian manifold, $\Omega\subset M$ a bounded connected domain, and $u:\overline{\Omega}\to\R$ an infinity harmonic function (viscosity solution of $\Delta_{\infty,g}u=0$ in $\Omega$).
    Then for every relatively compact subdomain $V\Subset\Omega$,
    $$
        \Lip(u,V)=\Lip(u,\partial V),
    $$
    where $\Lip(u,\cdot)$ denotes the Lipschitz constant.
\end{theorem}

\begin{proof}
The proof combines a chain of equivalences established in the literature.

First, by Corollary~4.5 of \cite{dragoni2013weak}, on a Riemannian manifold the infinity harmonic functions coincide with absolutely gradient minimizing extensions (AGME). Hence $u$ is an AGME.
Second, Theorem~4.7 in \cite{dragoni2013weak} states that for any Lipschitz function $u$, $\widehat{\Lip}(u,U)=\|\nabla u\|_{L^\infty(U)}$. Since $u$ is an AGME, it minimizes $\|\nabla u\|_{L^\infty(U)}$ among all functions with the same boundary values on $U$; therefore it also minimizes $\widehat{\Lip}(u,U)$, which is exactly the definition of a strong absolutely minimizing Lipschitz extension (SAMLE). Thus $u$ is a SAMLE.
Third, the weak Fubini property holds on any Riemannian manifold (Theorem~3.1 of \cite{dragoni2013weak}). Under this property, Juutinen and Shanmugalingam proved that in a length space the notions of AMLE and SAMLE are equivalent; see Theorem~2.9 in \cite{dragoni2013weak} or Corollary~3.5 therein. Consequently, $u$ is an AMLE.
Finally, the definition of an AMLE (see Definition~2.1 in \cite{juutinen2006equivalence} or Definition~2.1 in \cite{dragoni2013weak}) requires precisely that for every subdomain $V\subset\Omega$, $\Lip(u,V)=\Lip(u,\partial V)$. This completes the proof.
\end{proof}

\noindent{Xiaoshang Jin}\\
  School of mathematics and statistics, Huazhong University of science and technology, 430074, Wuhan, P.R. China
 \\Email address: jinxs@hust.edu.cn
 \\~\\
\noindent{Jiabin Yin}\\
 School of Mathematics and Statistics, Xinyang Normal University, 464000, Xinyang, P.R. China
	\\Email address: jiabinyin@126.com

\begin{thebibliography}{10}

\bibitem{adams2025full}
David~R Adams and Jie Xiao.
\newblock Full capacity--volumetry of sharp exp-integrability law.
\newblock {\em Journal of the London Mathematical Society}, 112(2):e70255,
  2025.

\bibitem{ballmann1985structure}
Ballmann, Werner and Brin, Misha and Eberlein, Patrick
\newblock Structure of manifolds of nonpositive curvature I.
\newblock {\em Annals of Mathematics}, 122(1):171-203, 1985.

\bibitem{bray2008capacity}
Hubert Bray and Pengzi Miao.
\newblock On the capacity of surfaces in manifolds with nonnegative scalar
  curvature.
\newblock {\em Inventiones mathematicae}, 172(3):459--475, 2008.

\bibitem{dragoni2013weak}
Federica Dragoni, Juan Manfredi, and Davide Vittone.
\newblock Weak fubini property and infinity harmonic functions in riemannian
  and sub-riemannian manifolds.
\newblock {\em Transactions of the American Mathematical Society},
  365(2):837--859, 2013.

\bibitem{F1982}
A Freire and R. Ma{\~n}\'{e}.
\newblock On the entropy of the geodesic flow in manifolds without conjugate points.
 \newblock {\em Inventiones Mathematicae}, 69:375-392, 1982

\bibitem{fogagnolo2022minimising}
Mattia Fogagnolo and Lorenzo Mazzieri.
\newblock Minimising hulls, p-capacity and isoperimetric inequality on complete
  riemannian manifolds.
\newblock {\em Journal of Functional Analysis}, 283(9):109638, 2022.

\bibitem{grigor1999isoperimetric}
Alexander Grigor’yan.
\newblock Isoperimetric inequalities and capacities on riemannian manifolds.
\newblock In {\em The Maz’ya Anniversary Collection: Volume 1: On
  Maz’ya’s work in functional analysis, partial differential equations and
  applications}, pages 139--153. Springer, 1999.

\bibitem{hijazi2020cheeger}
Oussama Hijazi, Sebasti{\'a}n Montiel, and Simon Raulot.
\newblock The cheeger constant of an asymptotically locally hyperbolic manifold
  and the yamabe type of its conformal infinity.
\newblock {\em Communications in Mathematical Physics}, 374(2):873--890, 2020.

\bibitem{jensen1993uniqueness}
Robert Jensen.
\newblock Uniqueness of lipschitz extensions: minimizing the sup norm of the
  gradient.
\newblock {\em Archive for Rational Mechanics and Analysis}, 123(1):51--74,
  1993.

\bibitem{jiang2018relative}
Renjin Jiang, Jie Xiao, and Dachun Yang.
\newblock Relative $\infty-$capacity and its affinization.
\newblock {\em Advances in Calculus of Variations}, 11(1):95--110, 2018.

\bibitem{jin2025lower}
Xiaoshang Jin.
\newblock {Lower bound for the first eigenvalue of $p-$Laplacian and
  applications in asymptotically hyperbolic Einstein manifolds}.
\newblock {\em Mathematical Research Letters}, 32(6):1947--1970.

\bibitem{jin2025relative}
Xiaoshang Jin.
\newblock The relative volume function and the capacity of sphere on
  asymptotically hyperbolic manifolds.
\newblock {\em Acta Mathematica Scientia}, 45(3):755--770, 2025.

\bibitem{jin2026sharp}
Xiaoshang Jin, Yao Wan, and Jie Xiao.
\newblock Sharp estimates for capacities-to-quermassintegrals within hyperbolic
  space.
\newblock {\em Submitted}, 2026.

\bibitem{jin2025sharply}
Xiaoshang Jin and Jie Xiao.
\newblock Sharply estimating hyperbolic capacities.
\newblock {\em arXiv preprint arXiv:2502.13651}, 2025.

\bibitem{juutinen1999eigenvalue}
Petri Juutinen, Peter Lindqvist, and Juan~J Manfredi.
\newblock The $\infty$-eigenvalue problem.
\newblock {\em Archive for rational mechanics and analysis}, 148(2):89--105,
  1999.

\bibitem{juutinen2006equivalence}
Petri Juutinen and Nageswari Shanmugalingam.
\newblock Equivalence of amle, strong amle, and comparison with cones in metric
  measure spaces.
\newblock {\em Mathematische Nachrichten}, 279(9-10):1083--1098, 2006.

\bibitem{Knieper1997} G. Knieper, \newblock{On the asymptotic geometry of nonpositively curved manifolds}.
 \newblock {\em GAFA Geom. funct. anal.} \textbf{7}, 755--782,  1997.


\bibitem{kruglikov1987capacity}
VI~Kruglikov.
\newblock Capacity of condensers and spatial mappings quasiconformal in the
  mean.
\newblock {\em Mathematics of the USSR-Sbornik}, 58(1):185--205, 1987.

\bibitem{lee1995spectrum}
John~M Lee.
\newblock The spectrum of an asymptotically hyperbolic einstein manifold.
\newblock {\em Communications in Analysis and Geometry}, 3(2):253--271, 1995.

\bibitem{li2025sharp}
Haizhong Li, Ruixuan Li, and Changwei Xiong.
\newblock Sharp upper bounds for the capacity in the hyperbolic and euclidean
  spaces.
\newblock {\em Advances in Nonlinear Analysis}, 14(1):20250068, 2025.

\bibitem{LW2010}
François Ledrappier, Xiaodong Wang.
\newblock An integral formula for the volume entropy with applications to rigidity.
\newblock {\em J. Differential Geom}. 85 (2010), no. 3, 461-477.

\bibitem{Ma1979}
A Manning.
\newblock Topological entropy for geodesic flows.
\newblock {\em  Annals of Mathematics}, 110(3):567–573, 1979.


\bibitem{mazurowski2023monotone}
Liam Mazurowski and Xuan Yao.
\newblock Monotone quantities for $ p $-harmonic functions and the sharp $ p
  $-penrose inequality.
\newblock {\em arXiv preprint arXiv:2305.19784}, 2023.

\bibitem{maz2013sobolev}
Vladimir Maz'ya.
\newblock {\em Sobolev spaces}.
\newblock Springer, 2013.

\bibitem{miao2024implications}
Pengzi Miao.
\newblock Implications of some mass-capacity inequalities.
\newblock {\em The Journal of Geometric Analysis}, 34(8):241, 2024.

\bibitem{miao2025mass}
Pengzi Miao.
\newblock Mass, capacitary functions, and the mass-to-capacity ratio: P. miao.
\newblock {\em Peking Mathematical Journal}, 8(2):351--404, 2025.

\bibitem{oronzio2025adm}
Francesca Oronzio.
\newblock Adm mass, area and capacity in asymptotically flat 3-manifolds with
  nonnegative scalar curvature.
\newblock {\em Communications in Contemporary Mathematics}, 27(09):2550011,
  2025.

\bibitem{oronzio2025area}
Francesca Oronzio.
\newblock Area, volume, and capacity in non-compact 3-manifolds with
  non-negative scalar curvature.
\newblock {\em International Mathematics Research Notices}, 2025(18):rnaf282,
  2025.


\bibitem{peigne2019entropy}
M. Peign\'e and A. Sambusetti.
\newblock Entropy rigidity of negatively curved manifolds of finite volume.
\newblock {\em Mathematische Zeitschrift}, 293:609--627, 2019.

\bibitem{polya1951isoperimetric}
G~Polya and G~Szeg{\"o}.
\newblock Isoperimetric inequalities in mathematical physics.(am-27), 1951.

\bibitem{xia2024new}
Chao Xia, Jiabin Yin, and Xingjian Zhou.
\newblock New monotonicity for p-capacitary functions in 3-manifolds with
  nonnegative scalar curvature.
\newblock {\em Advances in Mathematics}, 440:109526, 2024.

\bibitem{xiao2016p}
Jie Xiao.
\newblock The $p$-harmonic capacity of an asymptotically flat $3$-manifold with
  non-negative scalar curvature.
\newblock In {\em Annales Henri Poincar{\'e}}, volume~17, pages 2265--2283.
  Springer, 2016.

\bibitem{xiao2017p}
Jie Xiao.
\newblock $P$-capacity vs surface-area.
\newblock {\em Advances in Mathematics}, 308:1318--1336, 2017.

\bibitem{yin2026sharp}
Jiabin Yin and Xingjian Zhou.
\newblock Sharp upper bounds of $p$-capacity in convex cone and asymptotically
  flat half-space.
\newblock {\em Letters in Mathematical Physics}, 116(2):34, 2026.

\end{thebibliography}
\end{document}